\documentclass[12pt]{article} 
\usepackage{subfiles}
\usepackage{amsmath,amssymb,amsthm}
\usepackage[left=3cm, right=3cm, top=1in, bottom=1in]{geometry}
\usepackage{framed}
\usepackage{mdframed}
\usepackage{graphicx}
\usepackage{bm}
\usepackage{theoremref}
\usepackage{enumerate}
\usepackage{enumitem}
\usepackage{fancyhdr}
\usepackage{color}
\usepackage{float}
\usepackage{dsfont}
\usepackage{mathtools}
\usepackage{breqn}
\usepackage{bbm}
\usepackage{mathrsfs}
\usepackage{amstext}
\usepackage{hyperref}
\usepackage{accents}
\usepackage{tikz}
\usepackage{caption}
\usepackage{subcaption}
\usepackage{setspace}
\usepackage{cases}

\hypersetup{
	colorlinks  = true,
	citecolor  = blue,
	linkcolor  = blue
}
\usepackage{pdfsync}
\usepackage{amsthm}
\usepackage{amssymb}
\usepackage{mathrsfs}

\usepackage{color}

\usepackage{verbatim}
\usepackage[utf8]{inputenc}
\usepackage[english]{babel}

\usepackage[numbers,comma,sort]{natbib}

\newcommand{\notinsubfile}[1]{}

\numberwithin{equation}{section}
\newtheorem{theorem}{Theorem}[section]
\newtheorem{corollary}[theorem]{Corollary}
\newtheorem{lemma}[theorem]{Lemma}
\newtheorem{proposition}[theorem]{Proposition}

\theoremstyle{remark}
\newtheorem{definition}[theorem]{Definition}
\newtheorem{remark}[theorem]{Remark}
\newtheorem{example}[theorem]{Example}

\let\oldproofname=\proofname
\renewcommand{\proofname}{\rm\bf{\oldproofname}}

\newcommand{\mcC}{\mathcal{C}}

\newcommand{\mcF}{\mathcal{F}}

\newcommand{\mcP}{\mathcal{P}}

\newcommand{\mcS}{\mathcal{S}}

\newcommand{\indic}{\mathds{1}}

\newcommand{\mbE}{\mathbb{E}}

\newcommand{\mbN}{\mathbb{N}}

\newcommand{\mbP}{\mathbb{P}}

\newcommand{\mbR}{\mathbb{R}}

\newcommand{\mbT}{\mathbb{T}}

\newcommand{\mbZ}{\mathbb{Z}}

\newcommand{\Csigma}{C_{\bm{\sigma}}}

\newcommand{\RN}[1]{%
  \textup{\uppercase\expandafter{\romannumeral#1}}%
}

\newcommand{\supp}{\text{supp}}
\newcommand{\eps}{\varepsilon}
\newcommand{\dd}{\mathop{}\!\mathrm{d}}

\newcommand{\tzero}{|_{t=0}}

\newcommand{\blue}[1]{{\color{blue} #1}}

\long\def\avi#1{{\color{blue}Avi:\ #1}}

\newcommand{\vertiii}[1]{{\left\vert\kern-0.25ex\left\vert\kern-0.25ex\left\vert #1 
		\right\vert\kern-0.25ex\right\vert\kern-0.25ex\right\vert}}

\makeatletter
\newcommand{\customlabel}[2]{%
	\protected@write \@auxout {}{\string\newlabel {#1}{{#2}{\thepage}{#2}{#1}{}}}%
	\hypertarget{#1}{#2\hspace{-0.14cm}}
}
\makeatother

\author{Avi Mayorcas\thanks{Institute of Mathematics, Technische Universität Berlin, Straße des 17. Juni, 135, 10623, Berlin.\hfill ~ \hfill \newline \texttt{amayorcas@tu-berlin.de;} ORCID ID: \href{https://orcid.org/
0000-0003-4133-9740}{0000-0003-4133-9740}}\hspace{0.3em} and Milica Toma\v sevi\'c\thanks{CMAP, CNRS, École polytechnique, Institut Polytechnique de
Paris, 91120 Palaiseau, France. \hfill ~ \hfill \newline \texttt{milica.tomasevic@polytechnique.edu}.
}}

\title{Blow-up for a Stochastic Model of Chemotaxis Driven by Conservative Noise on $\mbR^2$}
\date{\today}
\makeatletter
\let\runauthor\@author
\let\runtitle\@title
\makeatother
\begin{document}
	\maketitle
		\begin{abstract}
			\par 
			
			We establish criteria on the chemotactic sensitivity $\chi$ for the non-existence 
			of global weak solutions (i.e. \textit{blow-up} in finite time) to a stochastic Keller--Segel model with spatially inhomogeneous, conservative noise on $\mbR^2$. We show that if $\chi$ is sufficiently large then \emph{blow-up} occurs with probability $1$. In this regime our criterion agrees with that of a deterministic Keller--Segel model with increased viscosity. However, for $\chi$ in an intermediate regime, determined by the variance of the initial data and the spatial correlation of the noise, we show that \emph{blow-up} occurs with positive probability.
			%
			
			
			

		\end{abstract}
		%
		%
		%
		%
%
		\section{Introduction}
		In this work, we present criteria for non-existence of global solutions (that we will frequently refer to as \textit{finite time blow-up}) to a stochastic partial differential equation (SPDE) model of chemotaxis on $\mbR^2$. The model we consider,
		\begin{equation}\label{eq:TurbulentParEllKS}
			\begin{cases} 
				\dd u_t = \left(\Delta u_t -\chi \nabla \cdot (u_t\nabla c_t)\right)\dd t  +  \sqrt{2\gamma} \sum_{k=1}^\infty \nabla \cdot \left(\sigma_k u_t \right)\circ d W^k_t,&\text{ on }\mbR_+\times \mbR^2 ,\\
				-\Delta c_t = u_t, &\text{ on } \mbR_+\times \mbR^2 ,\\
				u\tzero =u_0 \in \mcP(\mbR^2), &\text{ on }\mbR^2,
			\end{cases}
		\end{equation}
		is based on the parabolic-elliptic Patlak--Keller--Segel model of chemotaxis ($\gamma=0$) with the addition of a stochastic transport term ($\gamma>0$), where $\{W^k\}_{k\geq 1}$ is a family of i.i.d. standard Brownian motions on a filtered probability space, $(\Omega,\mcF,(\mcF_t)_{t\geq 0},\mbP)$, satisfying the usual assumptions. Here $\mcP(\mbR^2)$ denotes the set of probability measures on $\mbR^2$. We will give detailed assumptions on the vector fields $\sigma_k:\mbR^2\rightarrow \mbR^2$ below (see~\ref{ass:Noise1}-\ref{ass:Noise3}), but for now simply stipulate that they are assumed to be divergence free and such that $\bm{\sigma}:=\{\sigma_k\}_{k\geq 1}\in \ell^2(\mbZ;L^\infty(\mbR^2))$.
		\par
		The noiseless model ($\gamma=0$) is a well known PDE system modelling chemotaxis: the collective movement of a population of cells (represented by its time-space density $u$) in the presence of an attractive chemical substance (represented by its time-space concentration $c$). The chemical sensitivity is encoded by the parameter $\chi>0$. The main particularity of the model is that solutions may become unbounded in finite time even though the total mass is preserved. This is the so-called blow-up in finite time and it occurs depending on the spatial dimension of the problem and the size of the parameter $\chi$. In particular, on $\mbR^2$ blow-up occurs in finite time  for $\chi>8\pi$, at $t=\infty$ for $\chi=8\pi$ see \citet{blanchet_carrillo_masmoudi_2007} and global existence holds for $\chi<8\pi$, see for example the survey by \citet{perthame_04}. For results in other dimensions we refer to \cite{bossy_talay_96,osaki_yagi_01,hillen_potapov_04,perthame_04}.
		
		
		Since the scenario described by the noiseless model often occurs within an external environment, it is natural to take into account additional environmental effects. In
		some cases, this can be done by coupling additional equations into the system, such as the
		Navier--Stokes equations of fluid mechanics \cite{winkler_19_2d,winkler_19_3d,lorz_12}. With particular relevance to our work we note the results of \cite{kiselev_xu_16,iyer_xu_zlatos_20} where it was shown that transport by sufficiently strong \textit{relaxation enhancing flows} can have a regularising effect on the Keller--Segel equation. However, for both modelling and analysis purposes it
		is also relevant to study the effect of random environments. These either model a rough background, accumulated errors in measurement or emergent
		noise from micro-scale phenomena not explicitly considered.\par
		The noise introduced in \eqref{eq:TurbulentParEllKS} is related to stochastic models of turbulence, \cite{darling_92_isotropic,leJan_85_isotropic,kraichnan_68,chetrite_delannoy_gawcedzki_07_integrable} and we refer to the monograph by \citet{flandoli_11_pertubations} for a broader overview of its relevance to SPDE models. Noise satisfying either our assumptions, or closely related ones, has been applied in a number of related settings; interacting particle systems, \cite{coghi_flandoli_16,delarue_flandoli_vincenzi_13,flandoli_gubinelli_priola_11_euler}; regularisation, stabilisation and enhanced mixing of general parabolic and transport PDE, \cite{gess2021stabilization,flandoli_galeati_luo_21_mixing,flandoli_gubinelli_priola_10_wellposed}, and with particular applications to the Keller--Segel and Navier--Stokes equations amongst others in \cite{flandoli_galeati_luo_21_delayed,flandoli_luo_21_highmode, galeati_20_convergence}.\\
		
		The motivation of the present work is to understand the persistence of \textit{blow-up} in the case of stochastic chemotactic models driven by conservative noise. Our main result is that if $ \chi >\ (1+\gamma)8\pi $ then \textit{finite time blow-up} occurs almost surely, while if $ \chi >(1+ \gamma V[u_0]
		C_{\bm{\sigma}})8\pi$ then \textit{finite time blow-up} occurs with positive probability.  Here $V[u_0]$ denotes one half the spatial variance of the initial data and $C_{\bm{\sigma}}$ indicates a type of Lipschitz norm of the vector fields $\bm{\sigma}$ and measures the spatial decorrelation of the noise. We refer to \eqref{eq:CsigmaFinite} for a precise definition. Furthermore, if $\chi$ satisfies either of the above conditions and \textit{blow-up} does occur then it must do so before a deterministic time $T^\ast>0$, (see Theorem \ref{th:main-result}).
		Note that when $\gamma=0$, we recover the usual conditions for blow-up of the deterministic equation, see \cite{perthame_04}.\\ Three interesting regimes emerge from our criteria, on the one hand, if we let $C_{\bm{\sigma}}$ increase to infinity, the second condition becomes the first and \textit{blow-up} must occur almost surely, albeit for larger and larger $\chi$. On the other hand, when $C_{\bm{\sigma}}$ is arbitrarily small (which is the case for spatially homogeneous noise) one again recovers the deterministic criterion. However, in the third regime, where the noise and initial variance are reciprocally of the same order, i.e. $V[u_0] C_{\bm{\sigma}}<1$, we are only able to show \textit{blow-up} with positive probability. It is an interesting question, that we leave for future work, to obtain more information on the probability of \text{blow-up} in this case.  See remark \ref{rem:ConditionDiscussion} for a longer discussion of these points.
		
		The study of blow-up of solutions to SPDEs is a large topic of which we only mention some examples. It was shown by \citet{bonder_groisman_09} that additive noise can eliminate global well-posedness for stochastic reaction-diffusion equations, while a similar statement has been shown for both additive and multiplicative noise in the case of stochastic non-linear Schrödinger equations by \citet{deBouard_Debussche_02_BlowUp,deBoaurd_Debussche_05_BlowUpMult}. In addition, non-uniqueness results for stochastic fluid equations have been studied by~\citet{hofmanova_zhu_zhu_23_stoch_SNS}~and~\citet{romito_14_dyadic}.\par 
		In the case of SPDE models of chemotaxis the study of \textit{blow-up} phenomena has begun to be considered and we mention here two very recent works, by \citet{flandoli_galeati_luo_21_delayed} and \citet{misiats2021global}. In \cite{flandoli_galeati_luo_21_delayed} the authors show that under a particular choice of the vector fields, $\bm{\sigma}$, a similar model to \eqref{eq:TurbulentParEllKS} on $\mbT^d$ for $d=2,\,3$ enjoys delayed \textit{blow-up} with $1-\eps$ after choosing $\gamma$ and $\bm{\sigma}$ w.r.t. $\chi$ and $\eps \in (0,1)$. In \cite{misiats2021global} the authors study global well-posedness and blow-up of a conservative model similar to \eqref{eq:TurbulentParEllKS} with a constant family of vector fields $\sigma_k(x) = \sigma$ and a single common Brownian motion. Translating their parameters into ours, they establish global well-posedness of solutions to \eqref{eq:TurbulentParEllKS}, with $\sigma_k(x)\equiv 1$ and for $\chi <8\pi$, as well as finite time blow-up when $\chi>(1+\gamma)8\pi$.\par
		The main contribution of this paper is the above mentioned \textit{blow-up} criterion for an SPDE version of the Keller--Segel model in the case of a spatially inhomogeneous noise term. 
		 To the best of our knowledge, this is a new result. An interesting point is that, unlike the deterministic criterion, it relates the chemotactic sensitivity with the initial variance, regularity and intensity of the noise term. In addition, we close the gap in \cite{misiats2021global}, as in the case of constant vector fields we show that finite time \textit{blow-up} occurs for $\chi>8\pi$ (see Remark~\ref{rem:ConditionDiscussion}).  In addition, we show that $\chi>(1+\gamma)8\pi$ cannot be a sharp blow-up threshold for all sufficiently regular initial data.
		 
		 Our technique of proof follows the deterministic approach by tracking a priori the evolution in time of the spatial variance of solutions to \eqref{eq:TurbulentParEllKS}. We derive an SDE satisfied by this quantity which we analyse both pathwise and probabilistically to obtain criterion for blow-up.
		\paragraph{Notation}
\begin{itemize}[leftmargin=*]
	\item For $n\geq 1$ and $p\in [1,\infty)$ (resp. $p=\infty$) we write $L^p(\mbR^2;\mbR^n)$ for the spaces of $p$ integrable (resp. essentially bounded) $\mbR^n$ valued functions on $\mbR^2$.\\
	For $\alpha\in \mbR$ we write $H^\alpha(\mbR^2;\mbR^n)$ for the inhomogeneous Sobolev spaces of order $\alpha$ - a full definition and some useful facts are given in Appendix \ref{app:SobSpaces}. \\
	For $k\geq 0$ and $\alpha \in (0,1)$ we write $C^k(\mbR^2;\mbR^n)$ for the $k$ continuously differentiable maps and $\mcC^{k,\alpha}(\mbR^2;\mbR^n)$ for the $k$ continuously differentiable maps with $\alpha$ H\"older continuous $k^{\text{th}}$ derivatives. \\
	When the context is clear we remove notation for the target space, simply writing $L^p(\mbR^2)$, $H^\alpha(\mbR^2)$. We equip these spaces with the requisite norms writing $\|\,\cdot\,\|_{L^p},\, \|\,\cdot\,\|_{H^\alpha}$ removing the domain as well when it will not cause confusion.
	\item We write $\mcP(\mbR^2)$ for the space of probability measures on $\mbR^2$ and for $m\geq 1$, $\mcP_m(\mbR^2)$ for the space of probability measures with $m$ finite moments. By an abuse of notation we write, for example, $\mcP(\mbR^2)\cap L^p(\mbR^2)$ to indicate the space of probability measures with densities in $L^p(\mbR^2)$.
\item For $\mu \in \mcP(\mbR^2)$ and when they are finite we define the following quantities:
\begin{align*}
	C[\mu]&:= \int_{\mbR^2} x \dd \mu(x),\\
	V[\mu] &:= \frac{1}{2}\int_{\mbR^2}|x-C[\mu]|^2 \dd \mu(x) = \frac{1}{2}\int_{\mbR^2}|x|^2\dd\mu(x)- \frac{1}{2}|C[\mu]|^2.
	\end{align*}
	Note that $V[\mu]$ is one half the usual variance, we define it in this way for computational ease.
	\item For $T>0$, $X$ a Banach space and $p\in [1,\infty)$ (resp. $p=\infty$) we write ${L^p_TX:=L^p([0,T];X)}$ for the space of $p$-integrable (resp. essentially bounded) maps $f:[0,T]\rightarrow X$. Similarly we write $C_TX:=C([0,T];X)$ for the space of continuous maps ${f:[0,T]\rightarrow X}$, which we equip with the supremum norm ${\|f\|_{C_TX}:= \sup_{t\in[0,T]}\|f\|_X}$. We define the function space $S_T := C_TL^2(\mbR^2)\cap L^2_TH^1(\mbR^2)$.
	\item We write $\nabla$ for the usual gradient operator on Euclidean space while for $k\geq 2$, $\nabla^k$ denotes the matrix of $k$-fold derivatives. We denote the divergence operator by $\nabla \cdot $ and we write $\Delta:= \nabla \cdot \nabla$ for the Laplace operator.
	\item If we write $a\,\lesssim\, b$ we mean that the inequality holds up to a constant which we do not keep track of. Otherwise we write $a\,\leq\, C b$ for some $C >0$ which is allowed to vary from line to line.
 \item Given $a,\,b\in \mbR$ we write $a\wedge b \coloneqq \min \{a,b\}$ and $a\vee b \coloneqq \max\{a,b\}$.
	\end{itemize}
	\paragraph{Plan of the paper} In Section \ref{sec:main-res} we give the precise assumptions on the noise term and formulate our main result. Then, in Section~\ref{sec:prel} we establish some important properties of weak solutions to \eqref{eq:TurbulentParEllKS} which are made use of in Section~\ref{sec:proofs} where we prove our main theorem. Appendix~\ref{app:SobSpaces} is devoted to a brief recap of the fractional Sobolev spaces on $\mbR^2$ along with some useful properties. Appendix~\ref{app:StratToIto} gives a sketch proof for the equivalence between \eqref{eq:TurbulentParEllKS} and a comparable It\^o equation. Finally, in Appendix~\ref{app:LocalExist}, for the readers convenience, we provide a relatively detailed proof of local existence of weak solutions in the sense of Definition~\ref{def:StratSol} below. 
	\section{Main result}
	\label{sec:main-res}
	Before stating our main results we reformulate \eqref{eq:TurbulentParEllKS} into a closed form and state our standing assumptions on the noise.
	
	 It is classical that $c$ is uniquely defined up to a harmonic function, hence it can be written as  $c = K \ast u$ with $ K(x) = -\frac{1}{2\pi}\ln\left(|x|\right)$. Therefore, from now on, for $t>0$, we work with the expression
	\begin{equation}
		\label{eq:GreensKernels}
		\nabla c_t(x):= \nabla K \ast u_t (x) = -\frac{1}{2\pi} \int_{\mbR^2} \frac{x-y}{|x-y|^2} u_t(y)\dd y.
	\end{equation}

	Throughout we fix a complete, filtered, probability space, $(\Omega,\mcF,(\mcF_t)_{t\geq 0},\mbP)$, satisfying the usual assumptions and carrying a family of i.i.d Brownian motions $\{W^k\}_{k\geq 1}$. Furthermore, we consider a family of vector fields $\bm{\sigma}:= \{\sigma_k\}_{k\geq 1}$, satisfying the following assumptions.
	\begin{enumerate}
	\item[\customlabel{ass:Noise1}{{(H1)}}] For $k\geq 1$, $\sigma_k:\mbR^2 \rightarrow \mbR^2$ are measurable and such that $	\sum_{k=1}^\infty \|\sigma_k\|_{L^\infty}^2 <\infty.$
	\item[\customlabel{ass:Noise2}{{(H2)}}] For every $k\geq 1$, $\sigma_k \in C^2(\mbR^2;\mbR^2)$ and $ \nabla \cdot \sigma_k =0.$
	\item[\customlabel{ass:Noise3}{{(H3)}}] Defining $q :\mbR^2\times \mbR^2 \rightarrow \mbR^2 \otimes \mbR^2$ by
			\begin{equation*}
				q^{ij}(x,y)= \sum_{k=1}^\infty\sigma_k^i(x)\sigma_k^j(y),\quad \forall \, i,j =1,\ldots,d, \, x,y\in \mbR^2;
			\end{equation*}
			\begin{enumerate}
				\item The mapping $(x,y)\mapsto q(x,y)=:Q(x-y)\in \mbR^2 \otimes \mbR^2$ depends only on the difference $x-y$. 
				\item $Q(0)=q(x,x)=\text{Id}$ for any $x \in \mbR^2$.
				\item We have $Q \in C^2(\mbR^2 ;\mbR^2\otimes \mbR^2)$ and $	\sup_{x\in \mbR^2} |\nabla^2Q(x)|<\infty.$ 
			\end{enumerate}
			%
		\end{enumerate}
		\begin{remark}\label{rem:NoiseLip}
			For $\bm{\sigma}$ satisfying Assumption \ref{ass:Noise3} it is possible to show that the quantity
			\begin{equation}\label{eq:CsigmaFinite}
				C_{\bm{\sigma}} := \sup_{x\neq y \in \mbR^2} \sum_{k=1}^\infty \frac{|\sigma_k(x)-\sigma_k(y)|^2}{|x-y|^2}
			\end{equation}
		is finite. See \cite[Rem. 4]{coghi_flandoli_16} for details. Note that due to \ref{ass:Noise3}-(b) one cannot re-scale $\bm{\sigma}$ so as to remove $\gamma$ from \eqref{eq:TurbulentParEllKS}.
		\end{remark}
		\begin{remark}\label{rem:covariance_def}
			It is important to note that one can instead specify the covariance matrix $Q$ first. In fact, due to \cite[Thm. 4.2.5]{kunita_97_stochasticFlows} any matrix valued map $Q:\mbR^{2}\times \mbR^2\rightarrow \mbR^{2}\otimes \mbR^2$ satisfying the analogue of \eqref{eq:CsigmaFinite},
			$$ \sup_{x\neq y \in \mbR^2}\sum_{i=1}^2 \frac{q^{ii}(x,x)-2q^{ii}(x,y)+ q^{ii}(y,y)}{|x-y|^2} <\infty$$
			can be expressed as a family of vector fields $\{\sigma_k\}_{k\geq 1}$ satisfying \ref{ass:Noise1}-\ref{ass:Noise3}. 
		\end{remark}
Analysis and presentations of vector fields satisfying these assumptions can be found in \cite{coghi_flandoli_16}, \cite[Sec. 5]{gess2021stabilization} and \cite{delarue_flandoli_vincenzi_13,flandoli_gubinelli_priola_11_euler}. For the reader's convenience we give an explicit example here in the spirit of Remark~\ref{rem:covariance_def}, based on \cite[Ex.~5]{coghi_flandoli_16}, but adapted to our precise setting.
\begin{example}\label{ex:covariance}
Let $f\in L^1(\mbR_+)$ be such that $\int_{\mbR_+} rf(r) \dd r =\pi^{-1}$ and $\int_{\mbR^2}|u|^
2f(|u|)\dd u <\infty$. Then, let $\Pi: \mbR\rightarrow M_{2\times 2}(\mbR)$ be the $2\times 2$-matrix valued map defined by,
\begin{equation*}
    \Pi(u) = (1-p)\text{Id} + (2p-1) \frac{u\otimes u}{|u|^2}, \quad \text{for } p\in [0,1].
\end{equation*}
Then, we define the covariance function,
\begin{equation*}
    Q(z) := \int_{\mbR^2} \begin{pmatrix}
    \cos(u\cdot z)&-\sin(u\cdot z)\\
   \sin(u \cdot z)& \cos (u \cdot z)
    \end{pmatrix} \Pi(u) f(|u|) \dd u.
\end{equation*}
Property \ref{ass:Noise3} a) is satisfied by definition, after setting $q(x,y):= Q(x-y)$. Since 
$$Q(0)= \int_{\mbR^2} \Pi(u)f(|u|)\dd y,$$
property \ref{ass:Noise3} b) is easily checked by moving to polar coordinates, making use of elementary trigonometric identities and the normalisation $\int_{\mbR_+} r f(r)\dd r = \pi^{-1}$. Finally, \ref{ass:Noise3} c) can be checked by a straightforward computation using smoothness of the trigonometric functions and the moment assumption on $f$.
\end{example}
		We now define our notion of weak solutions.
\begin{definition}\label{def:StratSol}
			Let $\chi, \gamma>0$. Then, given $u_0 \in \mcP(\mbR^2)\cap L^2(\mbR^2)$, we say that a weak solution to \eqref{eq:TurbulentParEllKS} is a pair $(u,\bar{T})$ where
\begin{itemize}
    \item $\bar{T}$  is an $\{\mcF_t\}_{t\geq 0}$ stopping time taking values in $\mbR_+\cup \{\infty\}$,
    \item For $T< \bar{T}$, $u$ is an $S_T\coloneqq C_TL^2 \cap L^2_T H^1$ valued random variable such that
$$\mbE\left[\|u\|^2_{L^\infty_TL^1}+\|u\|^2_{L^\infty_T L^2} + \|u\|^2_{L^2_TH^1}\right]<\infty.$$
\end{itemize}
In addition, for any $t \in [0,T]$, $\phi\in H^1(\mbR^2)$, $\mbP$-a.s. the following identities hold,
\begin{equation}\label{eq:StratEquation}
\begin{aligned}
\langle u_t,\phi\rangle = & \langle u_0,\phi\rangle-\int_0^t \left( \langle \nabla u_s,\nabla \phi\rangle -\chi \langle u_s (\nabla K\ast u_s),\nabla \phi\rangle\right) \dd s \\
&- \sqrt{2 \gamma} \sum_{k\geq 1} \int_0^t 
\langle \sigma_k u_s,\nabla \phi\rangle \circ \dd W_s^k.\\
\end{aligned}
\end{equation}

			%
			%
			%
			%
		\end{definition}
		In Appendix \ref{app:LocalExist} we detail a standard argument to show that there exists a deterministic, positive time $T>0$ such that $(u,T)$ is a weak solution in the above sense. This is due to the particular structure of the noise and we stress that in general the maximal time of existence may be random.
		Applying the standard It\^o-Stratonovich correction one can prove the following remark, a sketch is given in Appendix \ref{app:StratToIto}.
		\begin{remark}\label{rem:StratToIto}
			Let $(u,\bar{T})$ be a solution, in the sense of Definition \ref{def:StratSol}, to \eqref{eq:TurbulentParEllKS}. Then it also holds that $(u,\bar{T})$ is a solution to the following It\^o equation: For every $\phi \in H^1(\mbR^2)$, $t\in [0,\bar{T}]$, $\mbP$-a.s.
		\begin{equation}\label{eq:ItoSol}
				\begin{aligned}
					\langle u_t,\phi\rangle = \,& \langle u_0,\phi\rangle -\int_0^t \left((1+\gamma )\langle \nabla u_t,\nabla \phi\rangle - \chi \langle u_s(\nabla K\ast u_s),\nabla \phi\rangle\right) \dd s\\
					&\blue{-}\sqrt{2\gamma} \sum_{k\geq 1} \int_0^t 
					\langle \sigma_k u_s,\nabla \phi\rangle \dd W_s^k,
				\end{aligned}
			\end{equation}
		\end{remark}
\begin{remark}
It follows from Definition \ref{def:StratSol} and the standard chain rule, obeyed by the Stratonovich integral, that for $u$ a weak, Stratonovich solution to \eqref{eq:TurbulentParEllKS} and $F \in C^3(L^2(\mbR^2);\mbR)$,
\begin{equation}\label{eq:StratChainRule}
\begin{aligned}
F[u_t] = &F[u_0] + \int_0^t DF[u_s][\Delta u_s -\chi \nabla \cdot(u_s\nabla K\ast u_s )]\dd s \\
&+ \sum_{k=1}^\infty \int_0^t DF[u_s][\nabla \cdot (\sigma_k u_s)]\circ \dd W^k_s,
\end{aligned}
\end{equation}
where $DF[u_s][\varphi]$ denotes the Gateaux derivative of $F[u_s]$ in the direction $\varphi \in H^{1}(\mbR^2)$. An equivalent It\^o formula for non-linear functional of \eqref{eq:ItoSol} also holds, see for example \cite[Sec.~2]{pardoux_21_SPDE_book}.
\end{remark}
		\begin{remark}
			Note that under assumption \ref{ass:Noise1}, for any $T>0$ and any weak solution on $[0,T]$, the stochastic integral is well defined as an element of $L^2(\Omega\times [0,T];L^2(\mbR^2))\subset L^2(\Omega\times [0,T];H^{-1}(\mbR^2))$, since for any $t\in (0,T]$, we have
			\begin{align*}
				\mbE\left[\sum_{k=1}^\infty \int_0^t \|\nabla \cdot (\sigma_k(x)u_s(x))\|^2_{L^2}\dd s\right] &=\mbE\left[\sum_{k=1}^\infty \int_0^t \| \sigma_k(x)\cdot \nabla  u_s(x)\|^2_{L^2}\dd s\right] \\
				&\,\leq\, \sum_{k=1}^\infty\|\sigma_k\|_{L^\infty}^2 \mbE\left[\int_0^t \|\nabla u_s\|^2_{L^2}\dd s\right] <\infty.
			\end{align*}
		\end{remark}
		We are ready to state our main result. 
		\begin{theorem}[\textit{Blow-up in finite time}]
			\label{th:main-result}
			Let $\chi,\,\gamma>0$ and let $u_0\in \mcP_{2}(\mbR^2)\cap L^2(\mbR^2)$ be such that $\int x u_0(x)\dd x =0$. Assume $\bm{\sigma}=\{\sigma_k\}_{k\geq 1}$ satisfy \ref{ass:Noise1}-\ref{ass:Noise3}. Let $(u,\bar T)$ be a weak solution to \eqref{eq:TurbulentParEllKS}. Then 
			%
			\begin{enumerate}[label=\roman*)]
			\item \label{it:AS_BlowUp}under the condition
		\begin{equation}\label{eq:BlowUpCond_almostsure}
				\chi > (1+\gamma)8\pi,
			\end{equation}
			 we have
			 $$\mbP(\bar T < T^\ast_1 )=1, $$
			for $T^\ast_1 := \frac{4\pi V[u_0]}{\chi-(1+\gamma)8\pi}$.
			    \item \label{it:Prob_BlowUp} under the condition
			    \begin{equation}\label{eq:BlowUpCond_positiveProb}
				\chi >(1+ \gamma V[u_0]
				C_{\bm{\sigma}})8\pi
			\end{equation}
			we have 
			$$\mbP(\bar T < T^\ast_2 )>0,$$
			for $T^\ast_2 := \frac{\log (\chi -8\pi) - \log\left(\chi -V[u_0]8\pi \gamma C_{\bm{\sigma}} -8\pi\right)}{2\gamma C_{\bm{\sigma}}}$.
			%
			
			\end{enumerate}
			%
			%
			%
			%
			%
			%
		\end{theorem}
		\begin{remark}\label{rem:ConditionDiscussion}
		\begin{itemize}[leftmargin=*]
		    \item If $V[u_0]C_{\bm{\sigma}}>1$ and $\chi$ satisfies \eqref{eq:BlowUpCond_positiveProb} then $\chi$ also satisfies \eqref{eq:BlowUpCond_almostsure}, in which case blow-up occurs almost surely before $T^*_1$. This has relevance to the setting of \cite{flandoli_galeati_luo_21_delayed} in which a model similar to \eqref{eq:TurbulentParEllKS} is considered on $\mbT^d$ for $d=2,\,3$ where formally $C_{\bm{\sigma}}$ can be taken arbitrarily large.

		    \item In the case $\Csigma=0$, which corresponds to noise that is independent of the spatial variable, criterion \eqref{eq:BlowUpCond_positiveProb} becomes $\chi>8\pi$ which is exactly the criterion for blow-up of solutions to the deterministic PDE. Applying Theorem \ref{th:main-result} one would only recover blow-up with positive probability in this case. However, using the spatial independence of the noise we can instead implement a change of variables, setting $v(t,x):= u(t,x-\sqrt{2\gamma}\sigma W_t)$. It follows from the Leibniz rule that $v$ solves a deterministic version of the PDE with viscosity equal to one. Hence, it blows up in finite time with probability one for $\chi>8\pi$. Note that in \cite{misiats2021global} a similar model was treated, amongst others, with spatially homogeneous noise and positive probability of blow-up was shown only for $\chi >(1+\gamma)8\pi$. 
 	 \item Observe that the second half of Theorem \ref{th:main-result} demonstrates that \eqref{eq:BlowUpCond_almostsure} cannot be a sharp threshold for almost sure global well-posedness of \eqref{eq:TurbulentParEllKS} for all initial data (or all families of suitable vector fields $\{\sigma_k\}_{k\geq 1}$). Given any $8\pi<\chi <(1+\gamma)8\pi$, initial data $u_0$ (resp. family of vector fields $\{\sigma_k\}_{k\geq 1}$) one can always choose suitable vector fields (resp. initial data) such that $\chi >(1+\gamma V[u_0]\Csigma)8\pi$ so that there is at least a positive probability that solutions cannot live for all time. However, the results of this paper leave open any quantitative information on this probability.
		    %
		\end{itemize}
		\end{remark}
		\begin{remark}\label{rem:TimeDiscussion}
			If we set $T^\ast := T^\ast_1 \wedge T^\ast_2$ then it is possible to show that $T^*$ respects the ordering of $ V[u_0]C_{\bm{\sigma}}$ and $1$. That is,
			$$T^* = \begin{cases}
				\frac{\log (\chi -8\pi) - \log\left(\chi -V[u_0]8\pi \gamma C_{\bm{\sigma}} -8\pi\right)}{2\gamma C_{\bm{\sigma}}}, & V[u_0]C_{\bm{\sigma}}<1,\\
				\frac{4\pi V[u_0]}{\chi-(1+\gamma)8\pi}, & V[u_0]C_{\bm{\sigma}}>1.
			\end{cases}$$
			As mentioned before, in the PDE case blow-up occurs, for $\chi>8\pi$, and weak solutions cannot exist beyond $T^* = \frac{4\pi V[u_0]}{\chi-8\pi}$. It follows that in all parameter regions both the threshold for $\chi$ and definition of $T^*$ in Theorem \ref{th:main-result} agree with the equivalent quantities in the limit $\gamma \rightarrow 0$.
		\end{remark}
		The proof of Theorem~\ref{th:main-result} is completed in Section~\ref{sec:proofs} after establishing some preliminary results in Section~\ref{sec:prel}. The central point is to analyse an SDE satisfied by $t\mapsto V[u_t]$.
	\section{A Priori Properties of Weak Solutions}
	\label{sec:prel}
	The following lemma demonstrates that the expression $\nabla c_t := \nabla K \ast u_t$ is well defined Lebesgue almost everywhere.
	\begin{lemma}
		\label{lemma:bound-drift}
		Let $(u,\bar{T})$ be a weak solution to \eqref{eq:TurbulentParEllKS} in the sense of Definition~\ref{def:StratSol}. Then, there exists a $C>0$ such that for all $t\in (0,\bar{T}]$,
\begin{equation}
     \|\nabla c_t\|_{L^\infty}\leq C \|u_t\|^{\frac{1}{4}}_{L^1} \|u_t\|^{\frac{1}{2}}_{L^2}\|u_t\|^{\frac{1}{4}}_{H^1}.
\end{equation}

  %
	\end{lemma}
	\begin{proof}
First, applying \cite[Lem.~2.5]{nagai_11_global} with $q=3$ gives,
\begin{equation}\label{eq:nagain_bound_1}
    \|\nabla c_u\|_{L^\infty} \lesssim \|u\|^{\frac{1}{4}}_{L^1} \|u\|^{\frac{3}{4}}_{L^3}.
\end{equation}
Interpolation between $L^2(\mbR^2)$ and $L^\infty(\mbR^2)$ gives,
\begin{equation}\label{eq:nagain_bound_2}
    \|u\|_{L^3}\,\leq\, \|u\|^{\frac{2}{3}}_{L^2}\|u\|^{\frac{1}{3}}_{L^\infty}.
\end{equation}
Combined with the embedding $H^1(\mbR^2)\hookrightarrow \mcC^{0,0}(\mbR^2)$ (see Lemma~\ref{lem:SobEmbeddings}) the required estimate is obtained.

	\end{proof} 
\begin{remark}
    Note that the choice of $q=3$ in the proof of Lemma~\ref{lemma:bound-drift} and the resulting exponents are essentially arbitrary, the only restriction being that a non-zero power of $\|u_t\|_{L^p}$ for some $p \in [1,2)$ must be included in the right hand side. The choice of $L^1$ is convenient since we will shortly demonstrate that $\frac{\dd}{\dd t}\|u_t\|_{L^1}=0$ for all weak solutions.
\end{remark}	

 \begin{remark}\label{rem:ImprovedStratEq}
		Exploiting symmetries of the kernel $K$, \eqref{eq:GreensKernels} and following \cite{senba_suzuki_2002_weak}, we can write the advection term of \eqref{eq:StratEquation} in a different form that will become useful later on. We note that,
		\begin{equation}
			\label{eq:drift1}
			\langle u_s\nabla c_s,\nabla \phi\rangle = \iint_{\mbR^{4}} u_s(x)\nabla_xK(x-y)\cdot \nabla \phi(x)u_s(y) \dd y\dd x.
		\end{equation}
		Renaming the dummy variables in the double integral and applying Fubini's theorem, we also have
		\begin{equation}
			\label{eq:drift2}
			\langle u_s\nabla c_s,\nabla \phi\rangle = \iint_{\mbR^{4}} u_s(y)\nabla_yK(y-x)\cdot \nabla \phi(y)u_s(x) \dd y\dd x.
		\end{equation}
		Combining \eqref{eq:drift1} and \eqref{eq:drift2} gives
		\begin{equation*}
			\langle u_s\nabla c_s,\nabla \phi\rangle = \frac{1}{2} \iint_{\mbR^{4}} u_s(x) u_s(y) (\nabla_xK(x-y)\cdot \nabla \phi(x) + \nabla_yK(y-x)\cdot \nabla \phi(y))\dd y\dd x.  \end{equation*}
		Therefore, in view of \eqref{eq:GreensKernels} we may re-write $\langle u_s\nabla c_s,\nabla \phi\rangle$ as 
		\begin{equation}\label{eq:ImprovedStratEq}
			\begin{aligned}
				\langle u_s\nabla c_s,\nabla \phi\rangle = & 
				-\frac{1}{ 4\pi} \iint_{\mbR^{4}} \frac{(\nabla \phi(x)-\nabla \phi(y))\cdot (x-y)}{|x-y|^{2}}u_s(x)u_s(y) \dd y \dd x 
			\end{aligned}
		\end{equation}
	\end{remark}
In order to prove our main result we will need to manipulate the zeroth, first and second moments of weak solutions. To do so we define a family of radial, cut-off functions, indexed by $\varepsilon \in (0,1)$ such that for some $C>0$
\begin{equation}\label{eq:CutOff}
\Psi_\eps(x) = \begin{cases}
	1, & \text{ for } |x|< \eps^{-1},\\
	0, & \text{ for } |x|>2\eps^{-1},
\end{cases}
\quad \|\nabla \Psi_\eps\|_{L^\infty}\,\leq\, C \eps, \quad \|\nabla^2 \Psi_\eps\|_{L^\infty} \,\leq\, C\eps^2.
\end{equation}
For any family of cut-off functions satisfying \eqref{eq:CutOff}, it is straightforward to show that there exists a $C>0$ such that
\begin{equation}
\label{eq:cutoff-sups}
\begin{aligned}
	\sup_{\substack{x\in \mbR^2\\\eps \in (0,1)}}|\nabla^2(x\Psi_\eps(x))|\,\,\leq\, C, \quad \sup_{\substack{x\in \mbR^2\\\eps \in (0,1)}}|\nabla^2(|x|^2\Psi_\eps(x))|\,\,\leq\, C.
\end{aligned}
\end{equation}
Note also that since $\supp (\nabla \Psi_\eps) = \supp (\Delta \Psi_\eps) = B_{2\eps^{-1}}(0)\setminus B_{\eps^{-1}}(0)$, then 
\begin{equation}
\label{eq:cut-off-normL2}
\|\nabla \Psi_\eps\|_{L^2}\,\leq\, C \eps^{1/2} \quad \text{ and } \quad
\|\Delta \Psi_\eps\|_{L^2} \,\leq\, C \eps^{3/2}.
\end{equation}
We start with sign and mass preservation.
\begin{proposition} \label{prop:sign-mass-preservation}
Let $(u,\bar{T})$ a weak solution to \eqref{eq:TurbulentParEllKS}. If $u_0\geq 0$ then $\mbP$-a.s.
\begin{enumerate}[label=\roman*)]
	\item $u_t\geq 0$ for all $t\in [0,\bar{T})$,\label{it:SignPreserve}
	\item $\|u_t\|_{L^1}= \|u_0\|_{L^1} =1$ for all $t\in [0,\bar{T})$. \label{it:L1Preserve}
\end{enumerate}
\end{proposition}
\begin{proof}
Let us define $$S[u_t]= \int u_t(x) u^-_t(x)\,\dd x = \|u_t^-\|^2_{L^2},$$ on $L^2(\mbR^2)$, where $u_t^- = u_t \indic_{\{u_t<0\}}$. The computations below can be properly justified by first defining an $H^1$ approximation of the indicator function, obtaining uniform bounds in the approximation parameter using that $u\in H^1$ and then passing to the limit using dominated convergence. For ease of exposition we work directly with $S[u_t]$ keeping these considerations in mind so that the following calculations should only be understood formally.

Applying \eqref{eq:StratChainRule} gives
\begin{equation}
	\label{eq:signStrat}
	\hspace{-1em}\begin{aligned}
		S[u_t] &= S[u_0] + 2\int_0^t \int_{\{u_s<0\}} \nabla u_s(x)\cdot \left(-\nabla u_s(x) +\chi \nabla c_s(x)u_s(x) \right)\,\dd x\,\dd s \\
		&\quad + \sqrt{2\gamma}\sum_{k=1}^\infty\int_0^t \int_{\{ u_s<0 \}}u_s(x)\nabla \cdot (\sigma_ku_s(x))\,\dd x \circ d W^k_s.
	\end{aligned}
\end{equation}
Regarding the stochastic integral term, using that $\nabla \cdot \sigma_k =0$, $u_s \big|_{\partial \{u_s<0 \}} =0$ and integrating by parts, we have 
\begin{align*}
	\int_{\{ u_s<0 \}}u_s(x)\nabla \cdot (\sigma_ku_s(x))\,\dd x =-\frac{1}{2}\int_{\{ u_s<0 \}}\nabla \cdot (u^2_s(x) \sigma_k) \,\dd x =0.
\end{align*}
Regarding the finite variation integral,
\begin{align*}
	\int_0^t \int_{\{u_s<0\}} \nabla u_s(x)\cdot (-\nabla &u_s(x) +\chi \nabla c_s(x)u_s(x) )\,\dd x\,\dd s&\\
	&=-\int_0^t\int_{\{u_s<0\}} |\nabla u_s(x)|^2\dd x + \chi \int_0^t\int_{\{u_s<0\}} \nabla u_s(x)\cdot u_s(x)\nabla c_s(x)\dd x,
\end{align*}
we apply Young's inequality in the second term, to give
\begin{align*}
	\chi \int_{\{u_s<0\}} \nabla u_s(x)\cdot u_s(x)\nabla c_s(x)\dd x &\,\leq\, \frac{\chi}{2\eps} \int_{\{u_s<0\}} |\nabla u_s(x)|^2\dd x + \|\nabla c_s\|^2_{L^\infty} \frac{\eps \chi}{2}\int_{\{u_s<0\}} |u_s(x)|^2\dd x.
\end{align*}
So choosing $\eps= \frac{\chi}{4}$ we have,
\begin{align*}
	-\int_{\{u_s<0\}} |\nabla u_s(x)|^2\dd x + \chi \int_{\{u_s<0\}} &\nabla u_s(x)\cdot u_s(x)\nabla c_s(x)\dd x \\
	&\,\leq\, -\frac{1}{2}\int_{\{u_s<0\}} |\nabla u_s(x)|^2\dd x + \frac{\|\nabla c_s\|^2_{ L^\infty}\chi^2}{8} \int_{\{u_s<0\}} |u_s(x)|^2\dd x
\end{align*}
Putting all this together in \eqref{eq:signStrat} and using that $\nabla u_s \in L^2(\mbR^2)$ for almost every $s\in [0,\bar{T}]$,
\begin{align*}
	S[u_t] &\,\leq\, S[u_0] + \frac{\chi^2}{4} \int_0^t \|\nabla c_s\|^2_{ L^\infty}S[u_s]\dd s.
\end{align*}
So, having in mind Lemma~\ref{lemma:bound-drift} and applying Gr\"onwall's inequality, we almost surely have
\begin{equation*}
	S[u_t] \,\leq\, 
	S[u_0] \exp\left(\frac{C\chi^2}{4} \|u\|^{\frac{1}{4}}_{L^\infty_{\bar{T}}L^1} \|u\|^{\frac{3}{4}}_{L^2_{\bar{T}}H^1} T \right),
\end{equation*}
where $C$ is the constant from Lemma~\ref{lemma:bound-drift}. Since $S[u_0] =0$ it follows that $\mbP$-a.s. $S[u_t]=0$ for all $t\in [0,\bar{T})$ which shows the first claim.

To show the second claim, for $\eps \in (0,1)$, we define $M_\eps[u_t] := \int_{\mbR^2} \Psi_\eps(x)u_t(x)\dd x$, where the cut-off functions $\Psi_\eps$ are given in \eqref{eq:CutOff}. Using the weak form of the equation and integrating by parts where necessary we see that
\begin{equation}\label{eq:MassEpsFormula}
	\begin{aligned}
		M_\eps[u_t] &= M_\eps[u_0] + (1+\gamma)\int_0^t \int_{\mbR^2} \Delta \Psi_\eps(x) u_s(x)\dd x \dd s\\
		&\quad - \chi \int_0^t \int_{\mbR^2} \nabla \Psi_\eps(x) \cdot \nabla c_s(x)u_s(x)\dd x \dd s\\
		& \quad - \sqrt{2\gamma } \sum_{k=1}^\infty \int_0^t \int_{\mbR^2} \nabla \Psi_\eps (x) \cdot (u_s(x)\sigma_k(x))\dd x  \dd W^k_s.
	\end{aligned}
\end{equation}
Applying the Cauchy-Schwartz inequality, the fact that the Itô integral disappears under the expectation and in view of \eqref{eq:cut-off-normL2}, there exists a $C>0$ such that
\begin{align*}
	\mbE\left[\int_{\mbR^2}\Psi_\eps(x)u_t(x)\dd x\right] & \,\leq\, \int_{\mbR^2} \Psi_\eps(x)u_0(x)\dd x \\
	&\quad+(1+\gamma) C\eps^{3/2} \mbE\left[\|u\|_{L^\infty_TL^2}\right] + \chi C\eps^{1/2}\mbE\left[ \|\nabla c\|_{L^2_TL^\infty}\|u\|_{L^\infty_TL^2}\right].
\end{align*}
Applying Fatou's lemma,
\begin{align*}
	\mbE\left[\int_{\mbR^2} u_t(x) \dd x\right] &\,\leq\, \liminf_{\eps \rightarrow 0} \mbE\left[\int_{\mbR^2}\Psi_\eps(x)u_t(x)\,\dd x\right] \,\leq\, \int_{\mbR^2} u_0(x)\dd x.
\end{align*}
Hence $\int_{\mbR^2} u_t(x)\dd x<\infty$ $\mbP$-a.s. for every $t\in [0,\bar{T})$. We may now apply dominated convergence to each term in \eqref{eq:MassEpsFormula}. In particular, stochastic dominated convergence is used for the last term on the right hand side. Thus, to obtain almost sure convergence all the limits should be taken up to a suitable subsequence. Finally, noting that $\Delta \Psi_\eps$ and $\nabla \Psi_\eps$ converge to zero pointwise almost everywhere, we conclude
\begin{equation*}
	M[u_t] =\int_{\mbR^2}u_t(x)\dd x = \int_{\mbR^2} u_0(x)\dd x.
\end{equation*}
In combination with the first statement of the lemma this proves the second claim.
\end{proof}
The following corollary to Proposition~\ref{prop:sign-mass-preservation} will be crucial to obtaining our central contradiction in the proof of Theorem~\ref{th:main-result}.
\begin{corollary}\label{cor:PositiveVar}
Let $(u,\bar{T})$ be a weak solution to \eqref{eq:TurbulentParEllKS}. Then for any $f:\mbR^2\rightarrow \mbR$ such that $f>0$ Lebesgue almost everywhere and any $t\in [0,\bar{T})$,
\begin{equation*}
	\int_{\mbR^2} f(x)u_t(x)\dd x >0 \quad \mbP\text{-a.s.}
\end{equation*}
\end{corollary}
\begin{proof}
We first show that any weak solution must have positive support. Let us fix an almost sure realisation of the solution, then chose any $t\in [0,\bar{T})$ and assume for a contradiction that, $u_t(\omega)$ is supported on a set of zero measure. However, since $\|u_t(\omega)\|_{L^1}=1$, we find that for any $p> 1$,
\begin{equation*}
	1= \int_{\mbR^2} u_t(x,\omega)\dd x \,\leq\, \left(\int_{\mbR^2}|u_t(x,\omega)|^p\dd x \right)^{\frac{1}{p}}\left(\int_{\supp(u_t(\omega))}1\,\dd x \right)^{\frac{p-1}{p}} = 0,
\end{equation*}
which is a contradiction. Since $f$ is assumed to be strictly positive, Lebesgue almost surely, the conclusion follows.
\end{proof}
In the following proposition, we derive the evolution for the center of mass and the variance of a weak solution to \eqref{eq:TurbulentParEllKS}.
\begin{proposition}\label{prop:BddMoments}
Let us assume that $u_0\in L^2(\mbR^2)\cap \mcP(\mbR^2)$ is such that
$V[u_0] <\infty$. Then for any weak solution to \eqref{eq:TurbulentParEllKS} in the sense of Definition \ref{def:StratSol}, $\mbP$-a.s. for any $t\in [0,\bar{T})$,
\begin{align}
	&C[u_t]= 
 -\sqrt{2\gamma} \sum_{k=1}^\infty \int_0^t \int_{\mbR^2} \sigma_k(x)u_s(x)\dd x \dd W^k_s, \label{eq:CoMIdentity}\\
	&\begin{aligned}V[u_t] =&V[u_0]+ \left(2 (1+\gamma)- \frac{\chi}{4\pi}\right)t - \frac{1}{2}|C[u_t]|^2\\
		&-\sqrt{2\gamma} \sum_{k\geq 1} \int_0^t \int_{\mbR^2} x \cdot \sigma_k(x)u_s(x)\dd x \dd W^k_s\end{aligned} \label{eq:VarIdentity}
\end{align}
\end{proposition}
%
%
%
\begin{proof}
Without loss of generality we may assume  that $C[u_0] = \int_{\mbR^2}x u_0(x)\dd x =0$. Indeed, given a non-centred initial condition $\tilde{u}_0$ with $C[\tilde{u}_0]=c\neq 0$ one may redefine ${C[u_t]:= \int_{\mbR^2}(x-c)u_t(x)\dd x}$ whose evolution along weak solutions to \eqref{eq:TurbulentParEllKS} will again, using the argument given below, satisfy the identity \eqref{eq:CoMIdentity}.  The rest of our analysis therefore holds without further change.

Let $p\in \{1,2\}$ and we use the convention that for $p=2$, $x^p := |x|^2$. Since $x^p\Psi_{\eps}(x)$ is an $H^1(\mbR^2)$ function we may apply \eqref{eq:ItoSol} along with Remark~\ref{rem:ImprovedStratEq} and integrate by parts where necessary to give that
\begin{align}\label{eq:formula-x-p}
	\int x^p&\Psi_\eps(x)u_t(x)\dd x = \int_{\mbR^2} x^p\Psi_\eps(x)u_0(x)\dd x 
	+(1+\gamma)\int_0^t\int_{\mbR^2} \Delta (x^p\Psi_\eps(x))u_s(x)\dd x \dd s \nonumber \\
	&-\frac{\chi}{4\pi}\int_0^t \iint_{\mbR^4} \frac{(\nabla(x^p\Psi_\eps(x))-\nabla(y^p\Psi_\eps(y)))\cdot (x-y)}{|x-y|^2}u_s(x)u_s(y)\,\dd y\,\dd x \,\dd s \\
	& - \sqrt{2\gamma} \sum_{k\geq 1}\int_0^t \int_{\mbR^2} \nabla (x^p\Psi_\eps(x))\cdot \sigma_k(x)u_s(x)\dd x \dd W^k_s. \nonumber
\end{align}
From \eqref{eq:cutoff-sups} it follows that uniformly across $x\in \mbR^2$ and $\eps\in (0,1)$, $\Delta (x^p\Psi_\eps(x))$ is bounded and $\nabla (x^p\Psi_\eps(x))$ is Lipschitz continuous. Hence, using that $\|u_t\|_{L^1}=1$ for all $t\in[0,\bar{T}]$ there exists a $C>0$ such that, for all $\eps \in (0,1)$,
\begin{equation*}
	\mbE\left[\int_{\mbR^2}x^p\Psi_\eps(x)u_t(x)\,\dd x\right] \,\leq\, \int_{\mbR^2} x^p \Psi_\eps(x)u_0(x)\dd x + \left( (1+\gamma) + \frac{\chi}{4\pi} \right)tC.
\end{equation*}
Note that we may directly apply Lebesgue's dominated convergence to the initial data term, since $|x^p\Psi_\eps(x)u_0(x)|\,\leq\, |x^pu_0(x)|$ where the latter is assumed to be integrable. Now, let us for the moment take only $p=2$. Applying Fatou's lemma,
\begin{align*}
	\mbE\left[\int_{\mbR^2} |x|^2 u_t(x) \dd x\right] &\,\leq\, \liminf_{\eps \rightarrow 0} \mbE\left[\int_{\mbR^2}|x|^2\Psi_\eps(x)u_t(x)\,\dd x\right]\\ &\,\leq\, \int_{\mbR^2} |x|^2 u_0(x)\dd x + \left( (1+\gamma) + \frac{\chi}{4\pi} \right)tC<\infty.
\end{align*}
Hence $\int_{\mbR^2} |x|^2 u_t(x)\dd x<\infty$ $\mbP$-a.s. From Proposition \ref{prop:sign-mass-preservation}, $u_t$ is a probability measure on $\mbR^2$, so we have the bound
\begin{equation*}
	\int_{\mbR^2} |x| u_t(x)\dd x \,\leq\, \left(\int_{\mbR^2} |x|^2 u_t(x)\dd x\right)^{1/2}.
\end{equation*}
It follows that for $p\in \{1,2\}$, $\int_{\mbR^2} x^p u_t(x)\dd x<\infty$ $\mbP$-a.s. Since by definition we also have,
$$|x^p\Psi_\eps(x)u_t(x)|
\,\leq\, |x^pu_t(x)|,$$
as in the proof of Proposition~\ref{prop:sign-mass-preservation}, we may apply dominated convergence in each integral of~\eqref{eq:formula-x-p}. Using that for $p\in\{1,2\}$ and Lebesgue almost every $x,\,y\in \mbR^2$
\begin{align*}
	\lim_{\eps \rightarrow 0} \Delta (x^p\Psi_\eps(x)) &= 0, \\
	\lim_{\eps \rightarrow 0} \nabla (x^p\Psi_\eps(x)-y^p\Psi_{\eps}(y)) &= \begin{cases}
		0, & \text{ if } p= 1, \\
		2(x-y), & \text{ if } p=2,
	\end{cases} \\
	\lim_{\eps \rightarrow 0} \nabla (x^p\Psi_\eps(x))&=\begin{cases}
		1,&\text{ if } p=1, \\
		2x, & \text{ if } p=2.
	\end{cases}
\end{align*} 
we directly find the claimed identities for $C[u_t]$ and $\frac{1}{2}\int_{\mbR^2}|x|^2u_t(x)\dd x$. To conclude it only remains to note that
$$ V[u_t] = \frac{1}{2}
\int_{\mbR^2} |x|^2u_t(x)\dd x - \frac{1}{2}|C[u_t]|^2.$$
\end{proof}
\section{Proof of Theorem~\ref{th:main-result}}
\label{sec:proofs}
%
%
%
We will prove both statements by demonstrating that in each case the a priori properties of any weak solution proved in Proposition~\ref{prop:sign-mass-preservation} will be violated at some finite time, either almost surely or with positive probability. Furthermore, we will make use of the identities shown in Proposition~\ref{prop:BddMoments}. Notice that our proofs of both of these propositions rely heavily on assumptions~\ref{ass:Noise1}-\ref{ass:Noise3}.

To prove \ref{it:AS_BlowUp} let us assume that given $u_0 \in \mcP_2(\mbR^2)\cap L^2(\mbR^2)$ and any associated weak solution $(u,\bar{T})$ it holds that $\mbP(\bar{T}=\infty)>0$ and let us choose any $\omega \in \{\bar{T}=\infty\}$. We may in addition assume that $\omega$ is a member of the full measure set where the solution lies in $C_TL^2(\mbR^2)\cap L^2_TH^1(\mbR^2)$ for any $T>0$ and the set where the It\^o integral is well defined. Applying \eqref{eq:VarIdentity} of Proposition~\ref{prop:BddMoments}, for any $0<t<\infty$ and the above $\omega$ we have that
\begin{equation}\label{eq:Var_AS_Bnd}
\begin{aligned}
    V[u_t](\omega) 
    = & V[u_0] + \left(2(1+\gamma) - \frac{\chi}{4\pi}\right)t - \frac{1}{2}\left|C[u_t]\right|^2(\omega)\\
    &- \sqrt{2\gamma}\sum_{k\geq 1} \left(\int_0^t \int_{\mbR^2} x\cdot \sigma_k(x)u_s(x) \dd x \dd W_s^k\right)(\omega).
    \end{aligned} 
\end{equation}
Now since the stochastic integral is by definition a local martingale, by the Dambis--Dubin--Schwarz theorem \cite[Ch.~V. Thm. ~1.6]{revuz_yor_08}, there exists a random time change $t\mapsto q(t)$, with $q(t)$ being the quadratic variation of the local martingale at time $t$, and a real valued Brownian motion $B$ such that for all $t$ in the range of $q$,
\begin{equation*}
V[u_t](\omega)= V[u_0] + \left(2(1+\gamma) -\frac{\chi}{4\pi}\right)t -\frac{1}{2}|C[u_t]|^2(\omega)-\sqrt{2\gamma} B_{q(t)(\omega)}(\omega).
\end{equation*}
Either the range of $q$ is $[0,\infty)$ or $q$ is bounded. If $\omega$ is such that the first case holds, then there exists a $T>0$ such that $\sqrt{2\gamma}B_{q(T)(\omega)}(\omega)=V[u_0]$, at which point, since by assumption $\left(2(\gamma+1)-\frac{\chi}{4\pi}\right) T-\frac{1}{2}|C[u_t]|^2(\omega)<0$ for any $T>0$,  we have,
\begin{equation*}
    V[u_T](\omega)<0.
\end{equation*}
The latter is in contradiction with Corollary~\ref{cor:PositiveVar}. Alternatively, if $\omega$ is such that $t\mapsto q(t)(\omega)$ is bounded, then there exists a $B_\infty(\omega)\in \mbR$ such that $\lim_{t\nearrow \infty}B_{q(t)(\omega)}(\omega) = B_\infty(\omega)$, see cite[Ch.~5, Prop.~1.8]. In which case, for all $t>0$,
\begin{equation*}
    V[u_t](\omega) \,\leq\, V[u_0] +\left(2(1+\gamma) -\frac{\chi}{4\pi}\right)t -\frac{1}{2}|C[u_t]|^2(\omega)+\sqrt{2\gamma}\left(|B_\infty|(\omega) +  |B_\infty-B_{q(t)}|(\omega)\right).
\end{equation*}
Since the final term vanishes for large $t>0$ and using again the fact that $2(1+\gamma) -\frac{\chi}{4\pi}<0$ there exists a $T>0$ sufficiently large such that once more
\begin{equation*}
    V[u_T](\omega) <0.
\end{equation*}
Again this contradicts Corollary~\ref{cor:PositiveVar} and as such our initial assumption that $\mbP(\bar{T}=\infty)>0$ must be false. Hence, $\mbP(\bar{T}<\infty)=1$.

Finally, taking the expectation on both sides of \eqref{eq:Var_AS_Bnd} we see that for
$$t\geq T^*_1 :=  \frac{4\pi V[u_0]}{\chi-(1+\gamma)8\pi}$$
one has
\begin{equation*}
    \mbE[V[u_t]] \,\leq\, 0.
\end{equation*}
Since $V[u_t]$ is non-negative, we must in addition have $\mbP(\bar{T}< T^*_1) =1$.

To prove \ref{it:Prob_BlowUp} let us instead assume that given any suitable initial data, for the associated weak solution one has $\mbP(\bar{T}=\infty)=1$. Now taking expectations on both sides of \eqref{eq:VarIdentity} we have,
\begin{align}\label{eq:V[u]Identity}
\mbE\left[V[u_t]\right] &= V[u_0] +\left(2(1+\gamma) -\frac{\chi}{4\pi}\right)t - \frac{1}{2}\mbE\left[\left|C[u_t]\right|^2\right].
\end{align}
Using \eqref{eq:CoMIdentity} and It\^o's isometry we see that
\begin{align}
\mbE\left[\left|C[u_t]\right|^2\right] &= 2\gamma \sum_{k=1}^\infty \int_0^t\mbE \left[\left|\int \sigma_k(x) u_s(x)\dd x \right|^2\right]\dd s \nonumber\\
&=2\gamma  \int_0^t\mbE \left[\left(\iint \sum_{k=1}^\infty\sigma_k(x)\cdot \sigma_k(y) u_s(y) u_s(x)\dd x \dd y \right) \right]\dd s\label{eq:CoMMeanAndVar}
\end{align}
where we exchanged summation and expectation using dominated convergence and recalling that $u_s\in \mcP(\mbR^2)$ $\mbP$-a.s. for all $s\in [0,T]$ and applying~\ref{ass:Noise1}.

Now the game is to estimate $\sum_{k=1}^\infty\sigma_k(x)\cdot \sigma_k(y)$ from below, remembering  that $u_t\geq0$ $\mbP$-a.s. for all $t\in [0,T]$. From Remark~\ref{rem:NoiseLip},
\begin{equation*}
\sum_{k=1}^\infty |\sigma_k(x)-\sigma_k(y)|^2 \,\leq\, \Csigma |x-y|^2,\quad \text{ for all }x,\,y \in \mbR.
\end{equation*}
As a direct consequence and in view of \ref{ass:Noise3}-b)
\begin{align*}
C_{\bm{\sigma}} |x-y|^2 \geq \sum_{k=1}^\infty |\sigma_k(x)-\sigma_k(y)|^2 &= \sum_{k=1}^\infty [|\sigma_k(x)|^2 -2\sigma_k(x)\cdot \sigma_k(y)+ |\sigma_k(y)|^2] \\
&= 2\text{Tr}(Q(0))-2\sum_{k=1}^\infty \sigma_k(x)\cdot\sigma_k(y)\\
&= 4 -2\sum_{k=1}^\infty \sigma_k(x)\cdot\sigma_k(y).
\end{align*}
Rearranging the above inequality, gives
\begin{equation}\label{eq:SigDiagLowerBnd}
\sum_{k=1}^\infty \sigma_k(x)\cdot \sigma_k(y) \geq 2 - \frac{1}{2}\Csigma |x-y|^2.
\end{equation}
Plugging \eqref{eq:SigDiagLowerBnd} into \eqref{eq:CoMMeanAndVar} and the fact that $u_t\in \mcP(\mbR^2)$ $\mbP$-a.s. for all $t\in [0,T]$, we find
\begin{align*}
\mbE[|C[u_t]|^2] &\geq 4\gamma \int_0^t \mbE\left[\iint u_s(x)u_s(y)\dd x\dd y \right] \dd s - C_{\bm{\sigma}} \gamma \int_0^t \mbE\left[\iint |x-y|^2 u_s(x)u_s(y)\dd x \dd y \right]\dd s\\
&= 4\gamma t - C_{\bm{\sigma}} \gamma \int_0^t \mbE\left[\iint |x-y|^2 u_s(x)u_s(y)\dd x \dd y \right]\dd s.
\end{align*}
The integrand in the second term can easily be rewritten as
\begin{align*}
\iint |x-y|^2 u_s(y)u_s(x)\dd x \dd y &= 4V[u_s].
\end{align*}
Thus, we establish the lower bound
\begin{equation}\label{eq:CoMVarLwrBnd0}
\begin{aligned}
\mbE[|C[u_t]|^2] &\geq 4\gamma t - 4\gamma C_{\bm{\sigma}} \int_0^t\mbE[V[u_s]]\,\dd s.
\end{aligned}
\end{equation}
Therefore, inserting \eqref{eq:CoMVarLwrBnd0} into \eqref{eq:V[u]Identity}, we find
%
%
\begin{align*}
\mbE[ V[u_t]]&\,\leq\, V[u_0] + \left(2-\frac{\chi}{4\pi}\right)t + 2\gamma C_{\bm{\sigma}} \int_0^t \mbE[V[u_s]]\,\dd s.
\end{align*}
Applying Gr\"onwall lemma,
\begin{align*}
\mbE[V[u_t]] &\,\leq\, V[u_0] + \left(2-\frac{\chi}{4\pi}\right)t +2\gamma C_{\bm{\sigma}} \int_0^t \left(V[u_0] + \left(2-\frac{\chi}{4\pi}\right)s\right) e^{(t-s)2\gamma C_{\bm{\sigma}} }\,\dd s.
\end{align*}
Evaluating the integrals,
\begin{align*}
\mbE[V[u_t]]
\,\leq\, \left(V[u_0] - \frac{1}{2\gamma C_{\bm{\sigma}}} \left(\frac{\chi}{4\pi}-2\right)\right)e^{2\gamma C_{\bm{\sigma}} t} +\frac{1}{2\gamma C_{\bm{\sigma}}}\left(\frac{\chi}{4\pi}-2\right).
\end{align*}
Which implies that if $\chi >8\pi\left(1+\gamma C_{\bm{\sigma}} V[u_0]\right)$
and $t\geq T^*_2$, for 
$$ T^*_2 := \frac{1}{2\gamma C_{\bm{\sigma}}} \left(\log (\chi -8\pi) - \log\left(\chi -V[u_0]8\pi \gamma C_{\bm{\sigma}} -8\pi\right)\right),$$
then $$\mbE[V[u_t]]\,\leq\, 0.$$
This is again in contradiction with Corollary \ref{cor:PositiveVar} which shows $\mbP$-a.s. positivity of $V[u_t]$ for $u$ a weak solution to \eqref{eq:TurbulentParEllKS}. Hence, our initial assumption must have been false and so for any weak solution the probability of global existence must be strictly less than $1$. Furthermore, using again the fact that $V[u_t]$ is almost surely non-negative, we must in fact have, $\mbP(\bar{T}< T^\ast_2)>0$.

\hfill \qedsymbol{}

\section*{Appendix}
\appendix
\section{Sobolev Spaces on $\mbR^d$}\label{app:SobSpaces}
We include some useful definitions and lemmas concerning inhomogeneous Sobolev spaces on $\mbR^2$.
\begin{definition}
Let $\alpha \in \mbR$. The non-homogeneous Sobolev space $H^\alpha(\mbR^2)$ consists of the tempered distributions $u\in \mcS'(\mbR^2)$ such that $\hat{u}\in L^1_{\text{loc}}(\mbR^2)$ and 
\begin{equation*}
\|u\|_{H^\alpha} := \|\mcF^{-1}((1+|\,\cdot\,|^2)^{\alpha/2}\hat{u})\|_{L^2}\,<\,\infty,
\end{equation*}
where $\mcF^{-1}$ denotes the inverse Fourier transform.
\end{definition}
We recall that $H^\alpha(\mbR^2)$ is a Hilbert space for all $\alpha \in \mbR$ with inner products,
\begin{equation*}
\langle u,v\rangle_{H^\alpha} := \int_{\mbR^2} (1+|\xi|^2)^\alpha \hat{u}(\xi)\hat{v}(\xi)\dd \xi.
\end{equation*}
When $\alpha =1$ we have $\langle u,v\rangle_{H^1}= \langle u,v\rangle_{L^2}+\langle \nabla u,\nabla v\rangle_{L^2}$. It follows from the definition that for $\alpha_0<\alpha_1$,
\begin{equation}\label{eq:SobMonotone}
H^{\alpha_1}(\mbR^2)\hookrightarrow H^{\alpha_0}(\mbR^2).
\end{equation}
%
%
%



We also recall the following interpolation and embedding results.
\begin{lemma}[{\cite[Prop. 1.32 \& Prop.1.52]{bahouri_chemin_danchin_11}}]\label{lem:SobInterp}
For $\alpha_0 \,\leq\, \alpha \,\leq\, \alpha_1$, it holds that $H^{\alpha_0}\cap H^{\alpha_1}\hookrightarrow H^\alpha$. In particular, for all $\theta\in [0,1]$ and $\alpha = (1-\theta)\alpha_0 + \theta \alpha_1$, one has
\begin{equation}
\|u\|_{H^\alpha} \,\leq\, \|u\|^{1-\theta}_{H^{\alpha_0}}\|u\|^\theta_{H^{\alpha_1}}.
\end{equation}
\end{lemma}
\begin{lemma}[{\cite[Thm. 1.66]{bahouri_chemin_danchin_11}}]\label{lem:SobEmbeddings}
For $\alpha\in \mbR$, the space $H^\alpha(\mbR^2)$ embeds continuously into
\begin{itemize}
\item the Lebesgue space $L^p(\mbR^2)$, if $0<\alpha<d/2$ and $2\,\leq\, p\,\leq\, \blue{\frac{4}{2-2\alpha}}$;
\item the H\"older space $\mcC^{k,\rho}(\mbR^2)$, if $\alpha \geq \blue{1+k+\rho}$ for some $k\in \mbN$ and $\rho\in [0,1)$.  
\end{itemize}
\end{lemma}
\section{Stratonovich to It\^o correction}\label{app:StratToIto}
We briefly detail the necessary calculations to justify Remark~\ref{rem:StratToIto}. We refer to \cite[Sec.~2.2]{coghi_flandoli_16}, \cite[Sec.~2]{flandoli_luo_21_highmode} and \cite[Sec.~2.3]{galeati_20_convergence} for similar arguments. All equalities below should be interpreted in the weak sense.
\begin{lemma}
Let $(u,\bar{T})$ be a Stratonovich solution to \eqref{eq:TurbulentParEllKS}, in the sense of Definition \ref{def:StratSol}, with $\bm{\sigma}$ satisfying Assumptions \ref{ass:Noise1}-\ref{ass:Noise3}. Then $u$ also solves the It\^o SPDE,
\begin{equation*}
\begin{cases}
\dd u_t = ((1+\gamma)\Delta u_t + \chi \nabla \cdot (u_t\nabla c_t)) \,\dd t + \sqrt{2\gamma}\sum_{k=1}^\infty \nabla \cdot (\sigma_k u_t) \dd W^k_t,&\\
-\Delta c_t = u_t,&\\
u\tzero = u_0.&
\end{cases}
\end{equation*}
\end{lemma}
\begin{proof}
Repeating the caveat that all equalities should be understood after testing against suitable test functions, for all $k\geq 1$ and making use of \ref{ass:Noise1} to ensure the stochastic integrals are well defined, 
\begin{equation*}
\int_0^t \nabla \cdot (\sigma_k u_s)\circ \dd W^k_s = \int_0^t \nabla \cdot (\sigma_k u_s)\dd W^k_s 
\,+\,\frac{1}{2}\int_0^t \nabla \cdot (\sigma_k \dd [ u,W^k]_s),
\end{equation*}
where the process $s \mapsto [ u,W^k]_s$ denotes the quadratic covariation between $u$ and $W$. Using \eqref{eq:TurbulentParEllKS} we find
\begin{equation*}
[ u,W^k]_s = \sqrt{2\gamma} \nabla \cdot (\sigma_k u_s),
\end{equation*}
so that we have,
\begin{equation}
\sqrt{2\gamma}\int_0^t \nabla \cdot (\sigma_k u_s)\circ \dd W^k_s = \sqrt{2\gamma}\int_0^t \nabla \cdot (\sigma_k u_s)\dd W^k_s \,+\, \gamma\int_0^t \nabla \cdot (\sigma_k \nabla \cdot (\sigma_k u_s))\dd s.
\end{equation}
Summing over $k\geq 1$ and applying the Leibniz rule, we see that,
\begin{align*}
\sum_{k=1}^\infty \nabla \cdot (\sigma_k(x) \nabla \cdot (\sigma_k(x) u_s(x))) &= \sum_{i,j =1}^d \partial_i\partial_j (Q^{ij}(0)u_s(x)) - \nabla \cdot \left(\left(\sum_{k=1}^\infty \nabla \sigma_k(x) \cdot \sigma_k(x)\right)u_s(x)\right), 
\end{align*}
where, $Q^{ij}(0)= \sum_{k=1}^\infty \sigma_k^i(x) \sigma_k^j(x)$ for any $x\in \mbR^2$ and $\nabla \sigma_k \cdot \sigma_k$ is the vector field with components,
\begin{equation*}
(\nabla \sigma_k \cdot \sigma_k)^i = \sum_{j=1}^d (\partial_j\sigma^i_k)\sigma^j_k.
\end{equation*}
Applying the Leibniz rule once more, for $j=1,\ldots, d$, we see that
\begin{align*}
\sum_{k=1}^\infty \sum_{j=1}^d (\partial_j\sigma^i_k(x))\sigma^j_k(x) & = \sum_{j=1}^d \partial_j q^{ij}(x,x) - \sum_{k=1}^\infty \sigma_k^i(x) \nabla \cdot \sigma_k(x).
\end{align*}
By Assumptions \ref{ass:Noise2} and \ref{ass:Noise3} we have that $\nabla \cdot \sigma_k =0$ and $q^{ij}(x,x)=Q^{ij}(0)=\delta_{ij}$ from which it follows that $\sum_{k=1}^\infty (\nabla \sigma_k \cdot \sigma_k)^i=0$ for all $i=1,\ldots,d$ and that
\begin{equation*}
\sum_{i,j =1}^d \partial_i\partial_j (Q^{ij}(0)u_s(x)) = \Delta u_s,
\end{equation*}
which completes the proof.
\end{proof}
\section{Local Existence}\label{app:LocalExist}
%
%
In this section we give a sketched proof of local existence of weak solutions to \eqref{eq:TurbulentParEllKS}. The method of proof is well known and can be found in a general form in \cite{prevot_rockner_07}. In the case of \eqref{eq:TurbulentParEllKS} a similar proof of local existence was exhibited in \cite[Prop.~3.6]{flandoli_galeati_luo_21_delayed}. For the readers convenience we supply here a lighter version adapted to our particular setting.
\begin{theorem}\label{th:LocalExist}
Let $u_0 \in L^2(\mbR^2)$. Then there exists a pair $(u,\bar{T})$, with $\bar{T}$ deterministic, which is a weak solution to \eqref{eq:TurbulentParEllKS} in the sense of Definition \ref{def:StratSol}. Furthermore, $\mbP$-a.s. $u\in C([0,\bar{T}];L^2(\mbR^2))$.
%
%
%
\end{theorem}

We begin with a local a priori bound on solutions to \eqref{eq:TurbulentParEllKS}. 
\begin{lemma}\label{lem:LocalAPriori}
Let $u_0 \in \mcP(\mbR^2)\,\cap\, L^2(\mbR^2)$. Then there exists a $\bar{T}=\bar{T}(\|u_0\|_{L^2})>0$ and a $C>0$, such that for any weak solution $u$ to \eqref{eq:TurbulentParEllKS} on $[0,\bar{T}]$,
\begin{equation}\label{eq:EnergyEstimateFinal}
 \sup_{t \in [0,\bar{T}]}\|u_t\|^2_{L^2}+ \int_0^{\bar{T}} \|u_t\|^2_{H^1}\dd t < C,\quad \mbP \text{ - a.s.}.
\end{equation}
Furthermore, it holds that
\begin{equation}\label{eq:appendix_apriori_L1}
    \|u_t\|_{L^1} = \|u_0\|_{L^1}=1 \quad \text{for all } t\in [0,\bar{T}).
\end{equation}
\end{lemma}
\begin{remark}
Since the constant on the right hand side of \eqref{eq:EnergyEstimateFinal} is non-random it follows immediately that $\|u\|_{L^\infty_{\bar{T}}L^2} + \|u\|_{L^2_{\bar{T}}H^1} \in L^p(\Omega;\mbR)$ for any $p\geq 1$.
\end{remark}
\begin{proof}[Proof of Lemma \ref{lem:LocalAPriori}]
The identity \eqref{eq:appendix_apriori_L1} is shown by Proposition~\ref{prop:sign-mass-preservation} so that we are only required to obtain \eqref{eq:EnergyEstimateFinal}.

By assumption $u_t \in H^1(\mbR^2)$ for all $t\in [0,\bar{T}]$ and it satisfies \eqref{eq:StratEquation}. In particular the Stratonovich integral is well-defined for $\mbP$-a.e. $\omega \in \Omega$. Applying \eqref{eq:StratChainRule} to the functional $F[u_t]:= \|u_t\|^2_{L^2}$, we have the identity,
\begin{equation}\label{eq:AdHocIto}
\begin{aligned}
\|u_t\|^2_{L^2} & = \|u_0\|^2_{L^2} - 2\int_0^t \|\nabla u_s\|_{L^2}^2 \dd s + 2\chi \int_0^t \langle \nabla u_s, u_s\nabla c_s \cdot \nabla u_s \rangle \dd s\\
&\quad -\sqrt{2\gamma} \sum_{k=1}^\infty \langle u_s \sigma_k, \nabla u_s \rangle \circ \dd W_t^k.
\end{aligned}
\end{equation}
For the non-linear term, integrating by parts and using the equation satisfied by $c$,
\begin{align*}
|\langle u_s\nabla c_s,\nabla u_s \rangle| = \frac{1}{2}|\langle \nabla c_s, \nabla (u_s^2)\rangle| = \frac{1}{2} \|u_s\|^3_{L^3}.
\end{align*}
Then using the Sobolev embedding $H^{1/2}(\mbR^2)\hookrightarrow L^3(\mbR^2)$, real interpolation as given by Lemma~\ref{lem:SobInterp} and Young's inequality, for any $\eps >0$,
\begin{equation*}
 \|u_s\|^3_{H^{1/2}} \,\leq\, \|u_s\|^{\frac{3}{2}}_{H^1}\|u_s\|^{\frac{3}{2}}_{L^2} \,\leq\, \frac{3}{4 \eps} \|u_s\|^2_{H^1} + \frac{\eps}{4}\|u_s\|^6_{L^2} \,\leq\, \frac{3}{4\eps}\|\nabla u_s\|^2_{L^2} + \frac{3}{4\eps}\|u_s\|^2_{L^2} + \frac{\eps}{4}\|u_s\|^6_{L^2}.
\end{equation*}
Regarding the stochastic integral, since each $\sigma_k$ is divergence free, it follows that,
\begin{equation*}
|\langle u_s\sigma_k,\nabla u_s\rangle | = \frac{1}{2}|\langle \sigma_k, \nabla (u_s^2)\rangle| = \frac{1}{2}|\langle 1, \nabla \cdot (\sigma_k u_s^2)\rangle| =0.
\end{equation*}
So, choosing $\eps = \chi$, we find that $\mbP$-a.s., 
\begin{equation}\label{eq:EnergyEstimateMed}
\|u_t\|^2_{L^2} \,\leq\, \|u_0\|^2_{L^2} -\frac{5}{8} \int_0^t \|\nabla u_s\|^2_{L^2}\dd s + \frac{3}{4}\int_0^t \|u_s\|^2_{L^2}\dd s+\frac{\chi^2}{4}\int_0^t \|u_s\|^6_{L^2}\dd s. 
\end{equation}
That is $t\mapsto \|u_t\|_{L^2}$ satisfies the non-linear, locally Lipschitz, differential inequality,
\begin{equation*}
\frac{\dd }{\dd t} \|u_t\|_{L^2} \,\leq\, \frac{3}{4}\|u_t\|_{L^2} +\frac{\chi^2}{4} \|u_t\|^3_{L^2}, \quad \mbP\text{-a.s.}
\end{equation*}
By standard ODE theory and recalling that $u_0$ is non-random, there exists a strictly positive, but possibly finite time $\bar{T}(\|u_0\|_{L^2})$ and a deterministic constant $C>0$, such that,
\begin{equation*}
\sup_{t\in [0,\bar{T}]}\|u_t\|_{L^2} \,\leq\, C, \quad \mbP\text{-a.s.}
\end{equation*}
Coming back to \eqref{eq:EnergyEstimateMed} to obtain a bound on $\int_0^t \|\nabla u_s\|^2_{L^2} \dd s$ for $t\,\leq\, \bar{T}$ completes the proof of \eqref{eq:EnergyEstimateFinal}.
\end{proof}
\begin{definition} We say that a mapping $A:H^{1}(\mbR^2)\rightarrow H^{-1}(\mbR^2)$ is \emph{locally coercive}, \emph{locally \blue{weakly} monotone} and \emph{hemi-continuous} if the following hold;

\emph{Locally coercive:} there exists an $\alpha>0$ such that if $u \in H^1(\mbR^2)$ with $\|u\|_{H^1}\,\leq\, R$ for any $R>0$ there exists a $\lambda>0$ for which it holds that
\begin{equation}\label{eq:SPDELocalCoercive}
2\langle A(u),u\rangle + \alpha\|u\|^2_{H^1}\,\leq\, \lambda\|u\|^2_{L^2}.
\end{equation}
\emph{Locally \blue{weakly} monotone:} for any $R>0$ there exists a $\lambda>0$ such that for all $u,\,w\in H^{1}(\mbR^2)$ with $\|u\|_{H^1}\vee \|w\|_{H^1}\,\leq\, R$ 
\begin{equation}\label{eq:SPDELocalMonontone}
2\langle A(u)-A(w),u-w\rangle \,\leq\, \lambda \left(\blue{\|u-w\|_{L^2}} + \|u-w\|^2_{L^2}\right).
\end{equation}
\emph{Hemi-continuous:} for any $u,\,w,\,v \in H^1(\mbR^2)$ the mapping,
\begin{equation}\label{eq:SPDEHemiCts}
\mbR \ni \theta \mapsto \langle A(u+\theta w),v\rangle \in \mbR,
\end{equation}
is continuous.
\end{definition}
\begin{lemma}\label{lem:GoodProps} The operator $A:\mcP(\mbR^2)\cap H^{1}(\mbR^2) \rightarrow H^{-1}(\mbR^2)$ given by the mapping,
\begin{equation*}
A(u):= \Delta u - \chi \nabla \cdot (u \nabla c),
\end{equation*}
is locally coercive, locally \blue{weakly} monotone and hemi-continuous.
\end{lemma}
\begin{proof}
\emph{Local Coercivity:} Approximating $u$ by smooth compactly supported functions it follows that,
\begin{align*}
\langle A(u),u\rangle &= -\|\nabla u \|_{L^2} + \chi \langle u \nabla c, \nabla u \rangle.
\end{align*}
By H\"older's inequality, Young's inequality and Lemma \ref{lemma:bound-drift}, for any $\eps>0$
\begin{equation*}
|\langle u \nabla c, \nabla u \rangle| \,\leq\, \|u\|_{L^2}\|\nabla c\|_{L^\infty}\|\nabla u\|_{L^2} \,\leq\, \frac{1}{2\eps}\|\nabla u\|^2_{L^2} + \frac{\eps}{2} \|u\|^2_{L^2}\|\nabla u\|^2_{H^1}
\end{equation*}
So that under the assumption that $\|u\|_{H^1}\,\leq\, R$ and choosing $\eps>0$ sufficiently small, there exist $\alpha,\,\lambda(R)>0$ such that 
\begin{equation*}
2|\langle A(u),u\rangle| \,\leq\, -\alpha \|u\|^2_{H^1} + \lambda \|u\|^2_{L^2}.
\end{equation*}

\emph{Local Weak Monotonicity:} Let us introduce the notation $-\Delta c_u = u$, so that we have
\begin{align*}
    \langle A(u)-A(w),u-w\rangle = - \|\nabla(u-w)\|^2_{L^2} + \chi \langle u\nabla c_u -w\nabla c_w,\nabla(u-w)\rangle.
\end{align*}
Applying Cauchy--Schwarz followed by the triangle inequality, Young's product inequality and H\"older's inequality gives,
\begin{align*}
    \langle A(u)-A(w),u-w\rangle 
    \leq &  - \|\nabla(u-w)\|^2_{L^2} + \chi \|\nabla(u-w)\|_{L^2}\left(\|u\nabla c_{u-w}\|_{L^2} + \|\nabla c_w(u-w)\|_{L^2}\right)\\
    \leq & -\frac{1}{2}\|\nabla(u-w)\|^2_{L^2} + \chi^2\|u\nabla c_{u-w}\|^2_{L^2} + \frac{\chi^2}{2}\|\nabla c_w (u-w)\|^2_{L^2}\\
    \leq & \chi^2\,\left( \|u\|^2_{L^2} \|\nabla c_{u-w}\|^2_{L^\infty} +  \|\nabla c_w\|^2_{L^\infty}\|u-w\|^2_{L^2}\right).
\end{align*}
Making use of Lemma~\ref{lemma:bound-drift} and the assumptions that $\|u\|_{L^1}\vee \|w\|_{L^1}=1$ and $\|u\|_{H^1}\vee \|w\|_{H^1}\leq R$ we find the estimates,
\begin{equation*}
   \|\nabla c_{u-w}\|^2_{L^\infty}\,\lesssim\,  \|u-w\|^{\frac{1}{2}}_{L^1}\|u-w\|^{\frac{1}{2}}_{H^1}\|u-w\|_{L^2} \, \lesssim\, R^{\frac{1}{2}}  \|u-w\|_{L^2}
\end{equation*}
and
\begin{equation*}
    \|\nabla c_w\|^2_{L^\infty} \,\lesssim\,  \|w\|^{1/2}_{L^1}\|w\|^{1/2}_{H^1}\|w\|_{L^2}\|u-w\|^2_{L^2}\,\leq \, R^{\frac{3}{2}}.
\end{equation*}
Hence, again using the assumption $\|u\|_{L^2}\,\leq\, \|u\|_{H^1}\,\leq \,R$, we find,
\begin{equation*}
     \langle A(u)-A(w),u-w\rangle \lesssim \chi^2\left( R^{\frac{5}{2}} \|u-w\|_{L^2} +  R^{\frac{3}{2}}\|u-w\|^2_{L^2}\right),
\end{equation*}
which proves the claim.

\emph{Hemi-continuity:} Letting $u,\,v,\,w \in H^1(\mbR^2)$ and $\theta \in \mbR$, we have
\begin{equation*}
|\langle A(u+\theta w)-A(u),v\rangle| \,\leq\, \theta |\langle \nabla w,\nabla v\rangle| + \chi |\langle (u+\theta w)\nabla c_{u+\theta w}-u\nabla c_u,\nabla v\rangle.
\end{equation*}
The first term directly converges to $0$ as $\theta\rightarrow 0$. For the second term, after applying H\"older's inequality we see that we are required to control
\begin{equation*}
\|(u+\theta w) \nabla c_{u+\theta w}-u\nabla c_u\|^2_{L^2} \,\leq\, \theta \left(\|u\nabla c_w\|^2_{L^2} + \|w\nabla c_u\|^2_{L^2}\right)+ \theta ^2\|w\nabla c_w\|^2_{L^2},
\end{equation*}
which again directly converges to $0$ as $\theta \rightarrow 0$.
\end{proof}
\begin{lemma}\label{lem:BGoodProps}
For $\bm{\sigma}:=\{\sigma_k\}_{k\geq 1}$ satisfying \ref{ass:Noise1} and divergence free, the mapping,
\begin{equation*}
H^1(\mbR^2) \ni u \mapsto \sum_{k\geq 1} \nabla \cdot (\sigma_k u) \in L^2(\mbR^2),
\end{equation*}
is linear and strongly continuous.
\end{lemma}
\begin{proof}
Linearity is clear. Let $u,\,w \in H^{1}(\mbR^2)$, using the divergence free property of the $\sigma_k$,
\begin{equation*}
\left\|\sum_{k\geq 1} \nabla \cdot (\sigma_k (u-w))\right\|_{L^2} \,\leq\, \sum_{k\geq 1} \|\sigma_k \cdot \nabla (u-w)\|_{L^2} \,\leq\, \|\bm{\sigma}\|_{\ell^2L^\infty}\|u-w\|_{H^1}.
\end{equation*}
\end{proof}
\begin{proof}[Proof of Theorem \ref{th:LocalExist}]
The strategy of proof is to first define a finite dimensional approximation to \eqref{eq:TurbulentParEllKS} using a Galerkin projection, we project the solution and the non-linear term to a finite dimensional subspace of $L^2(\mbR^2)$. Using Lemma \ref{lem:GoodProps} and the linearity of the noise term it follows that this finite dimensional system has a global solution and using the same arguments as in the proof of Lemma \ref{lem:LocalAPriori}, there is a non-trivial interval $[0,\bar{T}]$ on which we have uniform control on this solution. By Banach--Alaoglu we can extract a convergent subsequence, whose limit, $u$, will be our putative solution to \eqref{eq:TurbulentParEllKS}. By linearity the noise term converges so it will remain to show that $A$ converges along this subsequence to $A(u)$ and that $u$ is a solution in the sense of Definition \ref{def:StratSol}. Finally we prove that weak solutions in the sense of Definition \ref{def:StratSol} are unique. 

For $N\geq 1$, let $H_N \subset L^2(\mbR^2)$ denote the finite dimensional subspace spanned by the basis vectors $\{e_k\}_{|k|\,\leq\, N}$ and $\Pi_N: L^2(\mbR^2)\rightarrow H_N$ be an orthogonal projection such that $\|\Pi^N f\|_{L^2}\,\leq\, \|f\|_{L^2}$. Then we consider the finite dimensional system of Stratonovich SDEs,
\begin{equation}\label{eq:FiniteDimStratSDE}
\begin{aligned}
\dd u_{N;t} &= \left(\Delta u_{N;t} + \chi \nabla \cdot (\Pi_N(u_{N;t}\nabla c_{N;t})) \right)\dd t \\
&\quad + \sqrt{2\gamma}\sum_{k=1}^\infty \Pi_N(\sigma_k \nabla u_{N;t})\circ \dd W^k_t\\
u_{N;0}&= \Pi_N u_{0}.
\end{aligned}
\end{equation}
It follows from \cite[Thm. 3.1.1]{prevot_rockner_07} and Lemma \ref{lem:GoodProps} that a unique, global solution exists for all $N\geq 1$. Furthermore, for each $N\geq 1$, $u^N$ is a smooth solution to a truncated version of \eqref{eq:TurbulentParEllKS} with smooth initial data and is such that for all $t>0$ it holds that $\|u_{N;t}\|_{L^1} = \|u_{N;0}\|_{L^1}=1$. It is readily shown that
$$\langle \nabla u_N,\Pi_N(u_N\nabla c_N)\rangle = \langle \nabla u_N,u_N\nabla c_N\rangle.$$
Hence, using the same arguments as in the proof of Lemma \ref{lem:LocalAPriori}, there exists a $\bar{T}\in (0,\infty)$ depending only on $\|u^N_0\|_{L^2}\,\leq\, \|u_0\|_{L^2}$ such that
\begin{equation*}
\sup_{N\geq 1} \mbE\left[\sup_{t\in [0,\bar{T}]}\|u^N_t\|^2_{L^2} + \int_0^{\bar{T}}\|u_t^N\|^2_{H^1}\,\dd t\right]<\infty.
\end{equation*}
We can therefore apply the Banach--Alaoglu theorem, \cite[Thm. 3.16 \& Thm. 3.17]{brezis_11}, to see that there exist sub-sequences $\{u^k\}_{k\geq 1},\, \{A(u^k)\}_{k\geq 1}$, a $u \in L^2(\Omega\times [0,\bar{T}];H^1(\mbR^2))$ and a $\xi \in L^2(\Omega\times [0,\bar{T}];H^{-1}(\mbR^2))$ such that
\begin{align*}
&u^k \rightharpoonup u \in L^2(\Omega\times [0,\bar{T}];H^1(\mbR^2))\\
&A(u^k) \rightharpoonup \xi \in L^2(\Omega\times [0,\bar{T}];H^{-1}(\mbR^2)).
\end{align*}
It follows from the first and Lemma \ref{lem:BGoodProps} that the stochastic integrals converge so it remains to show that $\xi = A(u)$. From the local monotonicity of $A$, for any $t\in (0,\bar{T}]$, $v\in L^2(\Omega\times [0,\bar{T}];H^1(\mbR^2))$ and $N\geq 1$
\begin{equation}\label{eq:MonotoneApplied}
\mbE\left[\int_0^t \langle A(u^N_s)-A(v_s),u^N_s-v_s\rangle \dd s\right] \,\lesssim\, \frac{\lambda}{2}\mbE\left[\int_0^t \left(\blue{\|u^N_s-v_s\|_{L^2}}+\|u^N_s- v_s\|^2_{L^2}\right) \dd s\right].
\end{equation}
Using the identity,
\begin{equation*}
\mbE\left[ \|u^N_t\|^2_{L^2} - \|u^N_0\|^2_{L^2}\right] = \mbE\left[\int_0^t \langle A(u^N_s),u^N_s\rangle \dd s\right],
\end{equation*}
which can be proved directly using the chain rule for Stratonovich integrals and the arguments of Lemma \ref{lem:LocalAPriori}, it is straightforward to show the inequality,
\begin{align*}
\mbE\left[\int_0^t \langle \xi_s,u_s\rangle \dd s\right] \,\leq\, \liminf_{N\rightarrow \infty} \mbE\left[\int_0^t \langle A(u^N_s),u^N_s\rangle \dd s\right].
\end{align*}
It follows, applying \eqref{eq:MonotoneApplied} in the final inequality, that for any $v\in L^2(\Omega\times [0,\bar{T}];H^1(\mbR^2))$,
\begin{align*}
\mbE\left[\int_0^t \langle \xi_s-A(v_s),u_s-v_s\rangle \dd s\right] &\,\leq\, \liminf_{N\rightarrow \infty} \mbE\left[\int_0^t \langle A(u^N_s)-A(v_s),u^N_s-v_s\rangle \dd s\right]\\
&\,\lesssim\, \frac{\lambda}{2}\liminf_{N\rightarrow \infty} \mbE\left[\int_0^t\left(\blue{\|u^N_s-v_s\|_{L^2}}+\|u^N_s-v_s\|^2_{L^2}\right)\dd s\right]
\end{align*}
Now, choosing $v = u -\theta w$ for some $\theta>0$ and $w\in L^2(\Omega\times [0,\bar{T}];H^1(\mbR^2))$, gives that
\begin{equation*}
\mbE\left[\int_0^t \langle \xi_s-A(u_s-\theta w_s), w_s\rangle \dd s\right] \,\leq\, \theta\frac{\lambda}{2} \mbE\left[\int_0^t \|w_s\|^2_{L^2}\dd s\right].
\end{equation*}
So applying \eqref{eq:SPDEHemiCts} and taking $\theta\rightarrow 0$ we finally find that,
\begin{equation*}
\mbE\left[\int_0^t \langle \xi_s-A(u_s), w_s\rangle \dd s\right]\,\leq\, 0,
\end{equation*}
for all $w\in L^2(\Omega\times [0,\bar{T}];H^1(\mbR^2))$ from which it follows that $\xi=A(u_s) \in L^2(\Omega\times [0,\bar{T}];H^{-1}(\mbR^2))$.

It follows that $u \in L^2(\Omega;L^\infty([0,\bar{T}];L^2(\mbR^2))) \cap L^2(\Omega\times [0,\bar{T}];H^1(\mbR^2))$ and satisfies \eqref{eq:StratEquation}. We now show that in fact, $u \in L^2(\Omega;C([0,\bar{T}];L^2(\mbR^2))$. To see this we recall that since $L^2(\mbR^2)$ is a Hilbert space, if $u_{t_k} \rightharpoonup u_t \in L^2(\mbR^2)$, and $\|u_{t_k}\|_{L^2}\rightarrow \|u_t\|_{L^2}\in \mbR$ one has
\begin{equation*}
\|u_{t_k}-u_t\|^2_{L^2} =\langle u_t-u_{t_k},u_t-u_{t_k}\rangle = \|u_t\|^2_{L^2}-2\langle u_t,u_{t_k}\rangle + \|u_{t_k}\|_{L^2}^2 \rightarrow 0.
\end{equation*}
From \eqref{eq:AdHocIto} it follows that given a sequence $t_k\rightarrow t$, $\|u_{t_k}\|_{L^2}\rightarrow \|u_t\|_{L^2}$. So it suffices to show that $u_{t_k}\rightharpoonup u_t\in L^2(\mbR^2)$. Let $h\in L^2(\mbR^2)$ be arbitrary, $\{h_n\}_{n\geq 1}\subset H^1(\mbR^2)$ be a sequence converging to $h$ strongly in $L^2(\mbR^2)$ and $\varepsilon>0$, $n_\varepsilon\geq 1$ be large enough such that,
\begin{equation*}
\sup_{t \in [0,T]} \|u_t\|_{L^2}|\|h-h_n\|_{L^2}\,\,\leq\,
\, \frac{\varepsilon}{2}, \quad \text{for all }n\geq n_\varepsilon.
\end{equation*}
Therefore we have
\begin{align*}
|\langle u_t,h\rangle -\langle u_{t_k},h\rangle|_L^2 &\,\leq\, |\langle u_t,h-h_{n_\varepsilon}\rangle|+|\langle u_t-u_{t_k},h_{n_\varepsilon}\rangle| + |\langle u_{t_k},h-h_{n_\varepsilon}\rangle|\\
&\,\leq\, \varepsilon + \|u_{t}-u_{t_k}\|_{H^{-1}}\|h_{n_\varepsilon}\|_{H^1}.
\end{align*}
By definition, for any weak solution $u_{t_k} \rightarrow u_t$ strongly in $H^{-1}(\mbR^2)$ and so conclude
\begin{equation*}
\limsup_{n\geq n_\varepsilon} |\langle u_t,h\rangle - \langle u_{t_k},h\rangle| \,\,\leq\, \,\varepsilon.
\end{equation*}
Since $\varepsilon>0$ was arbitrary we may conclude $u_{t_k}\rightarrow u_t \in L^2(\mbR^2)$ strongly. Furthermore, inspecting the proof we see that the modulus of continuity is deterministic and hence $u\in L^p(\Omega;C([0,\bar{T}];L^2(\mbR^2)))$ for any $p\geq 1$.
\end{proof}

\paragraph{Acknowledgements}
The authors warmly thank J. Norris, A. de Bouard and L. Galeati for helpful discussions and B. Hambly for useful comments on the manuscript.

The authors would like to express their gratitude to the French Centre National de Recherche Scientifique (CNRS) for the grant (PEPS JCJC) that supported this project. M.T. was partly supported by \textit{Fondation Mathématique Jacques Hadamard}.

Work on this paper was undertaken during A.M.'s tenure as INI-Simons Post Doctoral Research Fellow hosted by the Isaac Newton Institute for Mathematical Sciences (INI) participating in programme \emph{Frontiers in Kinetic Theory}, and by the Department of Pure Mathematics and Mathematical Statistics (DPMMS) at the University of Cambridge.  This author would like to thank INI and DPMMS for support and hospitality during this fellowship, which was supported by Simons Foundation (award ID 316017) and by Engineering and Physical Sciences Research Council (EPSRC) grant number EP/R014604/1.
\bibliography{Mayorcas} 
\bibliographystyle{abbrvnat}
\end{document}